\documentclass[12pt,centertags]{amsart}
\usepackage{amsmath,amstext,amsthm,a4,amssymb,amscd}
\usepackage[mathscr]{eucal}
\usepackage{mathrsfs}
\usepackage{epsf}
\numberwithin{equation}{section}

\def\mE{\mathcal{E}}
\def\mF{\mathcal{F}}

\def\mM{\mathcal{M}}

\newtheorem{thm}{Theorem}[section]

\newtheorem{prop}[thm]{Proposition}

\theoremstyle{definition}
\newtheorem{rem}[thm]{Remark}
\theoremstyle{definition}

\theoremstyle{definition}
\newtheorem{defn}[thm]{Definition}
\newcommand{\be}{\begin{eqnarray}}
\newcommand{\ee}{\end{eqnarray}}

\newcommand{\comment}[1]{}

\begin{document}

\title{Positive scalar curvature on foliations: the enlargeability}
 
\author{Weiping Zhang}

\address{Chern Institute of Mathematics \& LPMC, Nankai
University, Tianjin 300071, P.R. China}
\email{weiping@nankai.edu.cn}

\begin{abstract}  We generalize the famous result of Gromov and 	Lawson on the nonexistence of metric of positive scalar curvature on enlargeable manifolds to the case of foliations, without using index theorems on noncompact manifolds. 
￼
\end{abstract}

\maketitle

\setcounter{section}{-1}

\section{Introduction} \label{s0}

It has been an important subject in differential  geometry  to study when a smooth manifold carries a Riemannian metric of positive scalar curvature (cf. \cite[Chap. IV]{LaMi89}).  A famous result of Gromov and Lawson \cite{GL80}, \cite{GL83} states that an enlargeable   manifold (in the sense of \cite[Definition 5.5]{GL83})  does not carry a metric of positive scalar curvature.  In particular, there is no metric of positive scalar curvature on any torus, which is a classical result of Schoen-Yau \cite{SY79} and Gromov-Lawson \cite{GL80}. A generalization to foliations  of the Schoen-Yau and Gromov-Lawson result on torus has been given in \cite[Corollary 0.5]{Z17}.  
In this paper, we further extend the above result of Gromov-Lawson on enlargeable manifolds  to the case of foliations.

Let $F $ be an
integrable subbundle of the tangent vector bundle $TM$ of a closed smooth manifold $M$. 
Let $g^F$ be a
  Euclidean metric  on $F$, and
$k^F\in C^\infty(M)$  be  the associated leafwise scalar curvature  (cf. \cite[(0.1)]{Z17}).  For any covering manifold $\pi:\widetilde M\rightarrow M$, one has a lifted integrable subbundle with metric  $(\widetilde F,g^{\widetilde F})=(\pi^*F,\pi^*g^F)$. 

\begin{defn}\label{t0.1} 
One calls $(M, F)$ an enlargeable foliation if for any $\varepsilon>0$, there is a   covering manifold $\widetilde M\rightarrow M$ and a smooth map $f:\widetilde M\rightarrow S^{\dim M}(1)$ (the standard unit sphere), which is constant near infinity and has non-zero degree, such that for any $X\in \Gamma(\widetilde F)$, $|f_*(X)|\leq \varepsilon |X|$.
\end{defn}

When $F=TM$ and $  M$ is spin, this is the original definition of the enlargeability of $M$ due to Gromov and Lawson \cite{GL80}, \cite{GL83}.

The main result of this paper can be stated as follows.

\begin{thm}\label{t0.2}   Let $(M,F)$ be an enlargeable foliation. Then (i): if $T  M$ is spin, then there is no $g^F$ such that  $k^F>0$  over $M$; (ii): if $  F$ is spin, then there is no $g^F$ such that  $k^F>0$  over $M$. 

\end{thm}

When   $F=TM$, one recovers the classical theorem of Gromov-Lawson \cite{GL80}, \cite{GL83}  mentioned at the begining. 
In a recent paper  \cite{BH},  Benameur and Heitsch proved Theorem \ref{t0.2}(ii) under the condition that $(M,F)$ has a Hausdorff homotopy groupoid.

As a direct consequence of Theorem \ref{t0.2}(i), one obtains an alternate proof,  without using the families index theorem,  of  \cite[Corollary 0.5]{Z17} mentioned above (for the special  case where the integrable subbundle on torus is spin, this result is due to Connes, as was stated in \cite[p. 192]{Gr96}). 

If $M$ is enlargeable and $(M,F)$ carries a transverse Riemannian structure, then Theorem \ref{t0.2}(i) is trivial, as in this case, if there is $g^F$ with $k^F>0$ over $M$, then one can construct    $g^{TM}$ with $k^{TM}>0$ over $M$, which contradicts with the Gromov-Lawson theorem.  Thus, the main difficulty for Theorem \ref{t0.2} is that there might be no transverse Riemannian structure on $(M,F)$. This is similar to what happens  in \cite{Co86} and \cite{Z17}, where one   adapts the Connes fibration constructed in \cite{Co86} to overcome this kind of difficulty. 
 
Recall that we have proved geometrically in \cite{Z17} that if $M$ is oriented and there exists $g^F$ with $k^F>0$ over $M$, then  under the condition that either $TM$ or $F$ is spin,  one has $\widehat A(M)=0$.  The case where  $F$ is spin is a famous result of Connes \cite[Theorem 0.2]{Co86}. 
  
 Our proof of Theorem \ref{t0.2}   combines  the   methods in \cite{GL80}, \cite{GL83} and \cite{Z17}.  It is based on deforming (twisted) sub-Dirac operators on the Connes fibration over $\widetilde M$.  A notable difference with respect to \cite{GL83}, where the relative index theorem  on noncompact manifolds plays an essential role,  is that we will work with compact manifolds even for the noncompactly enlargeable situation. It will be carried out in Section \ref{s2}.

\section{Proof of Theorem \ref{t0.2}}\label{s2}
 
  In this section, we first prove in Section \ref{s2.1} the easier  case where $(M,F)$ is a compactly enlargeable foliation, i.e., the covering manifold $\widetilde M$ in Definition \ref{t0.1} is compact.  Then in Section \ref{s2.3} we show how to extend the arguments in Section \ref{s2.1} to the case where $\widetilde M$ is noncompact. 
 
\subsection{The case of compactly enlargeable foliations
}\label{s2.1} 

Let $F$ be an integrable subbundle of the tangent bundle $TM$ of an oriented closed manifold $M$.

Let $g^F$ be a metric on $F$ and  $k^F$ be the scalar curvature of $g^F$. 
Let $(E,g^E)$ be a Hermitian vector bundle on $M$ carrying a Hermitian connection $\nabla^E$. Let $R^E=(\nabla^E)^2$ be the curvature of $\nabla^E$. 

 For any $\varepsilon>0$, we say $(E,g^E,\nabla^E)$ verifies the leafwise $\varepsilon$-condition if for any $X,\,Y\in\Gamma(F)$, the following pointwise formula holds on $M$,
\begin{align}\label{1.1}
\left|R^E(X,Y)\right|\leq \varepsilon\,|X|\, |Y|. 
\end{align}

The following result extends slightly \cite[Theorem 0.1]{Z17} and \cite[Theorem 0.2]{Co86}.\footnote{The case where $F$ is spin is due to Connes, cf. \cite[p. 192]{Gr96}.}

\begin{thm}\label{t1.1}
If $k^F>0$ over $M$ and   either $TM$ or $F$ is spin, then there exists $\varepsilon_0>0$ such that if $(E,g^E,\nabla^E)$ verifies the leafwise $\varepsilon_0$-condition, then $\langle \widehat A(TM){\rm ch}(E),[M]\rangle=0$. 
\end{thm}

\begin{proof} The proof of this theorem is an easy modification of the proof given in \cite{Z17} for the case of $E={\bf C}|_M$. We only give a brief description, by following the notations given in \cite{Z17}. 
Let $\delta>0$  be  such that $k^F\geq \delta$ over $M$.
Without loss of generality, we may well assume that $\dim M$, ${\rm rk}(F)$ are divisible by $4$, and that $TM$, $F$ and  $TM/F$ are oriented with compatible orientations.  

We assume first that $TM$ is spin.

Following  \cite[\S 5]{Co86} (cf. \cite[\S 2.1]{Z17}), let $\pi:\mM\rightarrow M$ be the Connes
fibration over $M$ such that for any $x\in M$, $\mM_x=\pi^{-1}(x)$
is the space of Euclidean metrics on the linear space $T_xM/F_x$.
Let  $T^V\mM$ denote the vertical tangent bundle of the fibration
$\pi:\mM\rightarrow M$. Then it carries a natural metric
$g^{T^V\mM}$.

By  using the Bott connection
  on $TM/F$, which is leafwise flat, one  lifts $F$ to an integrable subbundle
$\mF$ of $T\mM$. 
  Then $g^F$   lifts to a Euclidean metric $g^\mF=\pi^*g^F $ on $\mF$.

Let $\mF_1^\perp\subseteq T\mM$ be a subbundle, which is  transversal to $\mF\oplus T^V\mM$,   such that we have a
splitting $T\mM=(\mF \oplus T^V \mM)\oplus\mF_1^\perp$. Then
$\mF_1^\perp$ can be identified with $T\mM/(\mF \oplus T^V \mM)$
and carries a canonically induced metric $g^{\mF_1^\perp}$.
We denote  $\mF_2^\perp=T^V\mM$.

Let $\mE=\pi^*E$ be the lift of $E$ which carries the lifted Hermitian metric $g^\mE=\pi^*g^E$ and the lifted Hermitian connection $\nabla^\mE=\pi^*\nabla^E$.  Let $R^\mE=(\nabla^\mE)^2$ be the curvature of $\nabla^\mE$.

For any $  \beta,\ \varepsilon>0$,  following  \cite[(2.15)]{Z17}, let $g_{\beta,\varepsilon}^{T\mM}$ be the   metric    on $T\mM$  defined by
the
  orthogonal splitting,
\begin{align}\label{2.20}\begin{split}
       T\mM =   \mF\oplus \mF^\perp_1\oplus \mF^\perp_2,  \ 
\  \  \
g^{T\mM}_{\beta,\varepsilon}= \beta^2   g^{\mF}\oplus\frac{
g^{\mF^\perp_1}}{ \varepsilon^2 }\oplus g^{\mF^\perp_2}.\end{split}
\end{align}

Now we replace the sub-Dirac operator constructed in \cite[(2.16)]{Z17} by the obvious twisted (by $\mE$) analogue 
\begin{align}\label{2.21}
D^\mE_{\mF\oplus\mF_1^\perp,\beta,\varepsilon}:\Gamma\left(S_{\beta,\varepsilon} (\mF\oplus\mF_1^\perp)\widehat\otimes\Lambda^*\left(\mF_2^\perp\right)\otimes\mE\right)
\longrightarrow
\Gamma\left(S_{\beta,\varepsilon}(\mF\oplus\mF_1^\perp)\widehat\otimes\Lambda^*\left(\mF_2^\perp\right)\otimes \mE\right),
\end{align}
where $S_{\beta,\varepsilon}(\cdot)$ is the notation for the spinor bundle determined by  $g^{T\mM}_{\beta,\varepsilon}$.

The analogue of \cite[(2.28)]{Z17} now  takes the form 
\begin{multline}\label{2.39} 
 \left(D^\mE_{\mF\oplus\mF_1^\perp,\beta,\varepsilon}\right)^2=-\Delta^{\mE,\beta,\varepsilon}+\frac{k^\mF}{4\beta^2}+\frac{1}{2\beta^2}\sum_{i,\,j=1}^{{\rm rk}(F)}R^\mE(f_i,f_j)c_{\beta,\varepsilon}\left(\beta^{-1}f_i\right)c_{\beta,\varepsilon}\left(\beta^{-1}f_j\right)
\\
+
O_R\left(\frac{1}{\beta}+\frac{\varepsilon^2}{\beta^2}\right)
,
\end{multline}
where $-\Delta^{\mE,\beta,\varepsilon}\geq 0$ is the corresponding Bochner Laplacian,   $k^\mF=\pi^*k^F\geq \delta$ and $f_1,\,\cdots,\,f_{{\rm rk}(F)}$ is an orthonormal basis of $(\mF,g^\mF)$.  Moreover, the analogue of \cite[(2.34)]{Z17} now takes the   form
\begin{align}\label{1.4}
{\rm ind}\left(P_{R,\beta,\varepsilon,+}^\mE\right) =\left\langle\widehat A(TM){\rm ch}(E),[M]\right\rangle. 
\end{align}

From (\ref{1.1}), (\ref{2.39}), (\ref{1.4}) and proceed as in \cite[\S 2.2 and \S 2.3]{Z17}, one gets Theorem \ref{t1.1} for the case where $TM$ is spin easily. 
As in \cite[\S 2.5]{Z17}, the same proof applies to give a geometric proof for the case where $  F$ is spin, with an obvious modification of the (twisted) sub-Dirac operators   (cf. \cite[(2.58)]{Z17}). 
\end{proof}

Now for the proof of Theorem \ref{t0.2},  one follows \cite{GL80}, \cite{LaMi89} and chooses a complex vector bundle $E_0$ over $S^{\dim M}(1)$ such that 
\begin{align}\label{1.6}
\left\langle {\rm ch}\left(E_0\right),\left[S^{\dim M}(1)\right]\right\rangle\neq 0. 
\end{align}

From Definition \ref{t0.1} and \cite[(5.8) of Chap. IV]{LaMi89}, one sees that  for any $\varepsilon>0$, one can find a compact covering $\widetilde M\rightarrow M$ and a map $f:\widetilde M\rightarrow S^{\dim M}(1)$ of non-zero degree such that $E= f^*(E_0)$ verifies the leafwise $\varepsilon$-condition. Thus, if there is $g^F$ with $k^F>0$ over $M$, then  by Theorem \ref{t1.1} and in view of either  \cite[Theorem 0.1]{Z17}  (in the case where $  M$ is spin) or \cite[Theorem 0.2]{Co86} (in the case where $  F$ is spin),  one has
\begin{multline}\label{1.7}
0=\left\langle\widehat A\left(T\widetilde M\right){\rm ch}  \left(E\right) ,\left[\widetilde M\right]\right\rangle
=({\rm rk}(E_0))\widehat A\left(\widetilde M\right)
\\
+\left\langle\widehat A\left(T\widetilde M\right)  f^*\left( {\rm ch}\left(E_0\right)-{\rm rk}\left(E_0\right)\right),\left[\widetilde M\right]\right\rangle
={\rm deg}(f) \,\left\langle{\rm ch}\left(E_0\right),S^{\dim M}(1)\right\rangle ,
\end{multline}
where the last equality comes from the definition of ${\rm deg}(f)$, as ${\rm ch}(E_0)-{\rm rk}(E_0)$ is a top form on $S^{\dim M}(1)$. This contradicts with  (\ref{1.6}) and  completes the  proof of Theorem \ref{t0.2} for compact $\widetilde M$.

\begin{rem}\label{t1.2}
Since  any torus $T^n$ is compactly enlargeable (cf. \cite[p. 303]{LaMi89}), the proof above already applies to give an alternate proof of \cite[Corollary 0.5]{Z17} on the nonexistence of any foliation  with metric of positive leafwise scalar curvature on $T^n$. 
\end{rem}

\subsection{The case where $\widetilde M$ is noncompact}\label{s2.3}

We will deal with the case where $F=TM$  in detail.  We will work with compact manifolds, thus giving a new proof of the Gromov-Lawson theorem \cite[Theorem 5.8]{GL83} in the case where $\widetilde M$ is noncompact. With this ``compact" approach it is easy to    prove  the foliation extension as in  Section \ref{s2.1}.
 
We assume  that $\widetilde M$ is noncompact.  To simplify the  notation, from now on we simply denote $\widetilde M$ by $M$, or rather $M_\varepsilon$ to emphasize the dependence on $\varepsilon$. The key point is that the geometric data on $M$ now comes from isometric  liftings of  geometric data on a compact manifold.

Thus for any $\varepsilon>0$, let $f_\varepsilon:  M_\varepsilon\rightarrow S^{\dim M}(1)$ be as in Definition \ref{t0.1}. Let $K_\varepsilon\subset  M_\varepsilon$ be a compact subset of $M_\varepsilon$ such that 
$f(  M_\varepsilon\setminus K_\varepsilon)=x_0$, where $x_0 $ is a (fixed) point on $S^{\dim M}(1)$.\footnote{Up to an isometry of $S^{\dim M}(1)$, one can always  assume that $x_0$ is fixed and does not depend on $\varepsilon$.}  Following \cite{GL83}, we take a compact hypersurface $H_\varepsilon$  in $M_\varepsilon\setminus K_\varepsilon$. We denote by $M_{H_\varepsilon}$ the compact manifold with boundary $H_\varepsilon$ containing $K_\varepsilon$.

Let $M_{H_\varepsilon}'$ be another copy of $M_{H_\varepsilon}$. We glue $M_{H_\varepsilon}$ and $M_{H_\varepsilon}'$ along $H_\varepsilon$ to get the double, which we denote by $\widehat M_{H_\varepsilon}$. Let $g^{T\widehat M_{H_\varepsilon}}$ be a metric on $T\widehat M_{H_\varepsilon}$ such that $g^{T\widehat M_{H_\varepsilon}}|_{M_{H_\varepsilon}}=g^{TM}|_{M_{H_\varepsilon}}$. The existence of $g^{T\widehat M_{H_\varepsilon}}$ is clear.\footnote{Here we need not assume that  $g^{T\widehat M_{H_\varepsilon}}$ is of product structure near $M_{H_\varepsilon}$.} Let $S(T\widehat M_{H_\varepsilon})$ denote the corresponding spinor bundle. 

We extend $f_\varepsilon:M_{H_\varepsilon}\rightarrow S^{\dim M}(1)$ to $f_\varepsilon:\widehat M_{H_\varepsilon}\rightarrow S^{\dim M}(1)$ by setting $f_\varepsilon(M_{H_\varepsilon}')=x_0$.

Let $(E_0,g^{E_0} )$ be a Hermitian vector bundle on $S^{\dim M}(1)$ verifying (\ref{1.6}) and carrying a Hermitian connection $\nabla^{E_0}$.   Let $(E_1={\bf C}^k|_{S^{\dim M}(1)},g^{E_1},\nabla^{E_1})$, with $k={\rm rk}(E_0)$, be the canonical Hermitian trivial vector bundle on $S^{\dim M}(1)$.  Let $v:\Gamma(E_0)\rightarrow \Gamma (E_1)$ be an endomorphism such that $v|_{x_0}$ is an isomorphism. Let $v^*: \Gamma(E_1)\rightarrow \Gamma(E_0)$ be the adjoint of $v$ with respect to   $g^{E_0}$ and $g^{E_1}$. Set 
\begin{align}\label{1.11}
V=v+v^*. 
\end{align}
Then the self-adjoint endomorphism $V:\Gamma(E_0\oplus E_1)\rightarrow \Gamma(E_0\oplus E_1)$ is invertible near $x_0$.

Let $(\xi,g^\xi,\nabla^\xi)=(\xi_0\oplus\xi_1,g^{\xi_0}\oplus g^{\xi_1},\nabla^{\xi_0}\oplus\nabla^{\xi_1})=(f_\varepsilon^*E_0\oplus f_\varepsilon^*E_1, f_\varepsilon^*g^{E_0}\oplus f_\varepsilon^*g^{E_1},f_\varepsilon^*\nabla^{E_0}\oplus f_\varepsilon^*\nabla^{E_1})$ be the  ${\bf Z}_2$-graded Hermitian vector bundle with Hermitian connection over $\widehat M_{H_\varepsilon}$ (here for simplicity, we do not make explicit the subscript $\varepsilon$ in $\xi$, $\xi_0$ and $\xi_1$). Let $R^\xi=(\nabla^\xi)^2$ be the curvature of $\nabla^\xi$.
Set $V_{f_\varepsilon}=f_\varepsilon^*V$. Then   
\begin{align}\label{1.10}
  \left[\nabla^\xi,V_{f_\varepsilon}\right]=0
\end{align}
on $M_{H_\varepsilon}'$.

Let $D^\xi:\Gamma(S(T\widehat M_{H_\varepsilon})\widehat\otimes\xi)\rightarrow \Gamma(S(T\widehat M_{H_\varepsilon})\widehat\otimes \xi)$ be the canonically defined (twisted) Dirac operator (cf. \cite{LaMi89}). Let $D^\xi_\pm: \Gamma((S(T\widehat M_{H_\varepsilon})\widehat\otimes\xi)_\pm)\rightarrow \Gamma((S(T\widehat M_{H_\varepsilon})\widehat\otimes \xi)_\mp)$ be the obvious restrictions, where $(S(T\widehat M_{H_\varepsilon})\widehat\otimes\xi)_+=S_+(T\widehat M_{H_\varepsilon})\otimes\xi_0\oplus S_-(T\widehat M_{H_\varepsilon})\otimes \xi_1$, while $(S(T\widehat M_{H_\varepsilon})\widehat\otimes\xi)_-=S_-(T\widehat M_{H_\varepsilon})\otimes\xi_0\oplus S_+(T\widehat M_{H_\varepsilon})\otimes \xi_1$. By the Atiyah-Singer index theorem \cite{ASI} (cf. \cite{LaMi89})  and \cite{GL83}, one has 
\begin{align}\label{1.12}
{\rm ind}\left(D^\xi_+\right)=\left\langle\widehat A\left(T\widehat M_{H_\varepsilon}\right)\left({\rm ch}\left(\xi_0\right)-{\rm ch}\left(\xi_1\right)\right),\left[\widehat M_{H_\varepsilon}\right]\right\rangle = ({\rm deg}(f_\varepsilon))\left\langle {\rm ch}\left(E_0\right),\left[S^{\dim M}(1)\right]\right\rangle,
\end{align}
where the last equality comes from the definition of ${\rm deg}(f_\varepsilon)$ (cf. \cite{GL83}).

Let $k^{TM}$ denote the scalar curvature of $g^{TM}$. We assume that  there is $\delta>0$ such that $k^{TM}\geq \delta$ over $M$. 

For any $\varepsilon>0$, let $D_\varepsilon^\xi: \Gamma(S(T\widehat M_{H_\varepsilon})\widehat\otimes\xi)\rightarrow \Gamma(S(T\widehat M_{H_\varepsilon})\widehat\otimes \xi)$ be the deformed operator defined by
\begin{align}\label{1.13}
D^\xi_\varepsilon =D^\xi +  {V_{f_\varepsilon} }  
. 
\end{align}

\begin{prop}\label{t1.3}
There is  $\varepsilon_0>0$ such that for any $0<\varepsilon\leq\varepsilon_0$, one has $\ker (D^\xi_\varepsilon)=\{0\}$. 
\end{prop}
\begin{proof} 

Recall that $x_0\in S^{\dim M}(1)$ is fixed and $V|_{x_0}$ is invertible. 
Let $U_{x_0}\subset S^{\dim M}(1)$ be a (fixed) sufficiently small open neighborhood of $x_0$  such that the following inequality holds on $U_{x_0}$,
\begin{align}\label{4.1}
V^2\geq \delta_1.
\end{align}
Let $\psi:S^{\dim 1}(1)\rightarrow [0,1]$ be a smooth function such that $\psi=1$ near $x_0$ and ${\rm Supp}(\psi)\subset U_{x_0}$.  Then $\varphi_\varepsilon=1-f_\varepsilon^*\psi$ is a smooth function on $M_\varepsilon$ (and thus on $M_{H_\varepsilon}$), which extends to a smooth function on $\widehat M_{H_\varepsilon}$ such that $\varphi_\varepsilon=0$ on $  M_{H_\varepsilon}'$. 

Following \cite[p. 115]{BL91}, let $\varphi_{\varepsilon,1},\,\varphi_{\varepsilon,2}:\widehat M_{H_\varepsilon}\rightarrow [0,1]$ be defined by 
\begin{align}\label{1.14}
\varphi_{\varepsilon,1}=\frac{\varphi_\varepsilon} {\left(\varphi_\varepsilon^2+(1-\varphi_\varepsilon)^2\right)^{\frac{1}{2}}},\ \ \ \varphi_{\varepsilon,2}=\frac{1-\varphi_\varepsilon} {\left(\varphi_\varepsilon^2+(1-\varphi_\varepsilon)^2\right)^{\frac{1}{2}}}.
\end{align}
Then $\varphi_{\varepsilon,1}^2+\varphi_{\varepsilon,2}^2=1$. Thus, for any $s\in \Gamma(S(T\widehat M_{H_\varepsilon})\widehat\otimes\xi)$, one has
\begin{align}\label{1.15}
\left\|D_\varepsilon^\xi s\right\|^2 = \left\|\varphi_{\varepsilon,1}D_\varepsilon^\xi s\right\|^2 +\left\|\varphi_{\varepsilon,2}D_\varepsilon^\xi s\right\|^2 ,
\end{align}
from which one gets
\begin{multline}\label{1.16}
\sqrt{2}\left\|D_\varepsilon^\xi s\right\| \geq \left\|\varphi_{\varepsilon,1}D_\varepsilon^\xi s\right\|+\left\|\varphi_{\varepsilon,2}D_\varepsilon^\xi s\right\|
\\
\geq  
\left\| D_\varepsilon^\xi \left(\varphi_{\varepsilon,1}s\right)\right\|+\left\| D_\varepsilon^\xi\left(\varphi_{\varepsilon,2} s\right)\right\|
-\left\|c\left(d\varphi_{\varepsilon,1}\right)s\right\|-\left\|c\left(d\varphi_{\varepsilon,2}\right)s\right\|,
\end{multline}
where we identify $d\varphi_{\varepsilon,i}$, $i=1,\,2$, with the gradient of $\varphi_{\varepsilon,i}$.  

Let $e_1,\,\cdots,\,e_{\dim M}$ be an orthonormal basis of $g^{T\widehat M_{H_\varepsilon}}$. Then by (\ref{1.13}), one has
\begin{align}\label{1.17}
\left(D_\varepsilon^\xi\right)^2=\left(D^\xi\right)^2+
\sum_{i=1}^{\dim M} c\left(e_i\right)\left[\nabla^\xi_{e_i},V_{f_\varepsilon}\right]+ {V_{f_\varepsilon}^2} . 
\end{align}

From    (\ref{1.17}), one has for $j=1,\, 2$ that
\begin{align}\label{1.18}
  \left\| D_\varepsilon^\xi\left(\varphi_{\varepsilon,j} s\right)\right\|^2 =  
   \left\| D^\xi\left(\varphi_{\varepsilon,j} s\right)\right\|^2
+ \sum_{i=1}^{\dim M} \left\langle c\left(e_i\right)\left[\nabla^\xi_{e_i},V_{f_\varepsilon}\right]\,\varphi_{\varepsilon,j}s,\,\varphi_{\varepsilon,j}s\right\rangle
+ \left\|\varphi_{\varepsilon,j}V_{f_\varepsilon}s\right\|^2. 
\end{align}

By the Lichnerowicz formula \cite{L63} (cf. \cite{LaMi89}), one has on $M_{H_\varepsilon}$ that 
\begin{align}\label{1.19}
  \left(D^\xi\right)^2 =-\Delta ^\xi+\frac{k^{TM}}{4}+\frac{1}{2}\sum_{i,\,j=1}^{\dim M}c\left(e_i\right)c\left(e_j\right) R^\xi\left(e_i,e_j\right),
\end{align}
where $\Delta^\xi$ is the corresponding Bochner Laplacian  and $k^{TM}\geq \delta$ by assumption.  

By Definition \ref{t0.1}  and proceeding as in \cite[(5.8) of Chap. IV]{LaMi89}, one finds on $M_{H_\varepsilon}$ that
\begin{multline}\label{1.20}
 \frac{1}{2}\sum_{i,\,j=1}^{\dim M}c\left(e_i\right)c\left(e_j\right) R^\xi\left(e_i,e_j\right)
+  \sum_{i=1}^{\dim M} c\left(e_i\right)\left[\nabla^\xi_{e_i},V_{f_\varepsilon}\right]
\\
= \frac{1}{2}\sum_{i,\,j=1}^{\dim M}c\left(e_i\right)c\left(e_j\right) 
f_\varepsilon^*\left(R^{E_0}\left({f_\varepsilon}_*e_i,{f_\varepsilon}_*e_j\right)\right)
+  \sum_{i=1}^{\dim M} c\left(e_i\right)f_\varepsilon^*\left(\left[\nabla^{E_0\oplus E_1}_{{f_\varepsilon}_*e_i},V\right]\right)
=O\left(\varepsilon \right). 
\end{multline}

On the other hand, for any $1\leq i\leq \dim M$,  one verifies that
\begin{align}\label{1.23}
  e_i\left(\varphi_\varepsilon\right)=-e_i \left(f_\varepsilon^*\psi\right)=-f_\varepsilon^*\left(\left({f_\varepsilon }_*e_i\right)(\psi)\right)=O(\varepsilon).
\end{align}
 
From (\ref{1.14}) and  (\ref{1.23}), one finds that for $i=1,\,2$,
\begin{align}\label{1.24}
  \left|c\left(d\varphi_{\varepsilon,i}\right)\right|=O(\varepsilon). 
\end{align}

From (\ref{1.10}), (\ref{4.1}),  (\ref{1.14}), (\ref{1.16}), (\ref{1.18})-(\ref{1.20}) and (\ref{1.24}), one deduces that there exists $\delta_2>0$ such that when $\varepsilon>0$ is sufficiently small, one has (compare with \cite[p. 1062]{Z17})
\begin{align}\label{1.21}
 \left\|D_\varepsilon^\xi s\right\|\geq\delta_2\|s\|, 
\end{align}
which completes the proof of Proposition \ref{t1.3}. 
\end{proof}

From Proposition \ref{t1.3}, one finds ${\rm ind}(D^\xi_+)=0$, which contradicts with (\ref{1.12}) where the right hand side is non-zero. Thus, there should be no $g^{TM}$ with $k^{TM}>0$ over $M$. This completes the proof of Theorem \ref{t0.2} for the case of $F=TM$ (which is the original Gromov-Lawson theorem \cite[Theorem 5.8]{GL83}), without using the relative index theorem on noncompact manifolds in \cite{GL83}. 

Now to prove Theorem \ref{t0.2}(i), one simply combines the method in Section \ref{s2.1} with the doubling and gluing tricks above. The details are easy to fill.   Theorem \ref{t0.2}(ii)  follows by modifying the sub-Dirac operator as in \cite[\S 2.5]{Z17}. 

The proof of Theorem \ref{t0.2} is completed. 

$\ $

\noindent{\bf Acknowledgments.} This work was partially supported by NNSFC.  The author would like thank the referee for   helpful suggestions.

\def\cprime{$'$} \def\cprime{$'$}
\providecommand{\bysame}{\leavevmode\hbox to3em{\hrulefill}\thinspace}
\providecommand{\MR}{\relax\ifhmode\unskip\space\fi MR }
\providecommand{\MRhref}[2]{%
  \href{http://www.ams.org/mathscinet-getitem?mr=#1}{#2}
}
\providecommand{\href}[2]{#2}

\end{document}